\documentclass[12pt,a4paper,reqno]{amsart}
\usepackage[english]{babel}
\usepackage[applemac]{inputenc}
\usepackage[T1]{fontenc}
\usepackage{palatino}
\usepackage{amsmath}
\usepackage{amssymb}
\usepackage{amsthm}
\usepackage{amsfonts}
\usepackage{graphicx}
\usepackage{color}
\input epsf

\newtheorem{theorem}{Theorem}[section]
\newtheorem{lemma}[theorem]{Lemma}
\newtheorem{corollary}[theorem]{Corollary}
\newtheorem{proposition}[theorem]{Proposition}

\theoremstyle{definition}
\newtheorem{definition}[theorem]{Definition}

\theoremstyle{remark}
\newtheorem{remark}[theorem]{Remark}

\numberwithin{equation}{section}

\DeclareMathOperator{\dimH}{dim_H}

\DeclareMathOperator{\proj}{proj}

\DeclareMathOperator{\Int}{int}

\newcommand{\pr}{\mathbb{R}}

\newcommand{\spt}{\text{spt}\, }
\newcommand{\Om}{\Omega}
\newcommand{\om}{\omega}
\newcommand{\hm}{\mathcal{H}}
\newcommand{\ga}{\gamma}
\newcommand{\dd}{\delta}
\newcommand{\be}{\beta}
\newcommand{\ob}{\overline{\beta}}
\newcommand{\al}{\alpha}
\newcommand{\oa}{\overline{\alpha}}

\newcommand{\x}{\times}
\renewcommand{\tilde}{\widetilde}

\begin{document}
\title[Projection Theorem and Geodesic Flows]
{Besicovitch-Federer projection theorem and geodesic flows on Riemann surfaces}

\author[R. Hovila]{Risto Hovila$^1$}
\address{Department of Mathematics and Statistics, P.O.Box 68,
         00014 University of Hel\-sin\-ki, Finland$^1$}
\email{risto.hovila@helsinki.fi}

\author[E. J\"arvenp\"a\"a]{Esa J\"arvenp\"a\"a$^2$}
\address{Department of Mathematical Sciences,  P.O. Box 3000,
         90014 University of Oulu, Finland$^{2,3}$}
\email{esa.jarvenpaa@oulu.fi$^2$}

\author[M. J\"arvenp\"a\"a]{Maarit J\"arvenp\"a\"a$^3$}
\email{maarit.jarvenpaa@oulu.fi$^3$}

\author[F. Ledrappier]{Fran\c cois Ledrappier$^4$}
\address{LPMA, UMR 7599, Universit\'e Paris 6, 4, Place Jussieu, 75252 Paris
         cedex 05, \linebreak France$^4$}
\email{fledrapp@nd.edu$^4$}

\thanks{We acknowledge the support of the Centre of
Excellence in Analysis and Dynamics Research funded by the Academy of Finland.
RH acknowledges the support from Jenny and Antti Wihuri Foundation.
FL acknowledges the partial support from NSF Grant
DMS-0801127.}


\begin{abstract}
We extend the Besicovitch-Federer projection theorem to transversal
families of mappings. As an application
we show that on a certain class of Riemann surfaces with constant negative
curvature and with boundary, there exist natural 2-dimensional measures
invariant under the geodesic flow having 2-dimensional supports such that their
projections to the base manifold are 2-dimensional but the supports of the
projections are Lebesgue negligible. In particular, the union
of complete geodesics has Hausdorff dimension 2 and is Lebesgue negligible.
\end{abstract}

\maketitle

\section{Introduction}\label{intro}

A {\it{pair of pants}} $S$ is  a 2-sphere minus three points endowed with a
metric of constant curvature $-1$ in such a way that the boundary consists of
three closed geodesics of length $a$, $b$ and $c$ called the {\it{cuffs}}. The
metric is uniquely determined by these three lengths. (For more details, 
see e.g. \cite{H}.) For each point $x$ in $S$, write $\Om _x$ for
the set of unit tangent vectors $v\in T_x^1S$ such that the geodesic ray
$\ga_v(t), t\geq 0,$ with initial condition $(x,v)$ never meets the boundary 
$\partial S$ of
$S$. The set $\Om_x$ is a Cantor set of dimension $\dd =\dd(a,b,c)$. The
number $\dd$ is an important geometric invariant of the pair of pants $S$: it
is the critical exponent of the Poincar\'e series of $\pi_1(S)$ and the
topological entropy of the geodesic flow on $T^1 S$ (cf. \cite{Su}). We will
recall in Section \ref{application} why the function $(a,b,c )\mapsto\dd$ is
real analytic. In particular, the function $a\mapsto\dd(a,a,a)$ is
continuous from $(0,\infty )$ onto the open set $(0,1)$. In a very similar
setting, McMullen (\cite {MM}) gives asymptotics  for $1-\dd(a,a,a)$ when
$a\to 0$ and for $\dd(a,a,a)$ when $a\to\infty.$

We are interested in the set
\begin{equation}\label{PiNW}
\begin{aligned}
C(S) := \{x\in S\mid \ &{\textrm {there exists }}v\in T^1_x S{\textrm{ such that }}\\
                  \ & v\in\Om_x {\textrm { and }} -v\in\Om_x\}.
\end{aligned}
\end{equation}
In other words, $C(S)$ is the set of points in complete geodesics in $S$. 
Let 
\begin{equation}\label{NW}
D(S) := \{(x,v)\in T^1S\mid x\in C(S),v\in\Om_x,-v\in\Om_x\}
\end{equation}
be the subset of $T^1S$ where the geodesic flow is defined for all 
$t\in\mathbb R$. Clearly, $\Pi(D(S))=C(S)$, where $\Pi:T^1S\to S$, 
$\Pi(x,v)=x$, is the canonical projection. 

We write $\mathcal L^l$ and $\mathcal H^s$ to denote 
the $l$-dimensional Lebesgue measure and the $s$-dimensional Hausdorff 
measure. For the Hausdorff dimension we use the notation $\dimH$. 

We consider the following theorem:

\begin{theorem}\label{negligible}
With the above notation,
\begin{itemize} 
  \item[-]  $\mathcal L^2(C(S))>0$ provided that $\dd > 1/2$ and 
  \item[-]  $\dimH C(S) = 1+ 2\dd$ and $\mathcal L^2(C(S))=0$ provided that
$\dd \leq 1/2$.
\end{itemize}
\end{theorem}

It is known that $\dimH(D(S))=1+2\dd$ (see 
Section~\ref{application}). Ledrappier and 
Lindenstrauss proved in \cite{LL} (see \cite{JJL} for a different proof) that
$\Pi$ does not diminish the Hausdorff dimension of
a measure which is invariant under the geodesic flow. 
The new part of our result is when $\dd $ is exactly
$1/2$. In that case, \cite{LL} implies that $\dimH C(S) = 2$, and we sharpen 
this by proving that $C(S)$ is Lebesgue negligible. 

The main technical part of
our paper is the following extension of Besicovitch-Federer projection theorem 
to transversal families of maps. (For the definition of transversality, 
see Definition \ref{transversality}.) We believe that Theorem \ref{main} is of 
independent interest (see for example \cite{OS}), and therefore 
we verify it in a more general setting  than needed for the purpose of proving 
Theorem \ref{negligible}.

\begin{theorem}\label{main}
Let $E\subset\mathbb R^n$ be $\mathcal H^m$-measurable with
$\mathcal H^m(E)<\infty$. Assume that $\Lambda\subset\pr^l$ is open and
$\{P_\lambda:\mathbb R^n\to\mathbb R^m\}_{\lambda \in \Lambda}$ is a transversal
family of maps. Then $E$ is purely
$m$-unrectifiable, if and only if $\mathcal H^m(P_\lambda(E))=0$ for
$\mathcal L^l$-almost all $\lambda\in\Lambda$.
\end{theorem}

In \cite{HJJL} we showed that on any Riemann surface with (variable) negative
curvature there exist 2-dimensional measures which are invariant under 
the geodesic flow and have singular
projections with respect to $\mathcal L^2$. The
measures are supported by the whole unit tangent bundle $T^1S$ and they are 
singular
with respect to $\mathcal H^2$ on $T^1S$. However, the measures constructed in 
this 
paper have 2-dimensional supports and they are absolutely continuous with 
respect to 
$\mathcal H^2$ on $T^1S$. Thus their singularity is due to the projection.  

The paper is organized as follows: In Section~\ref{results} we introduce the
notation and prove Theorem~\ref{main}. In
Section~\ref{application} we recall basic properties of the geodesic flow
on a pair of pants and prove Theorem~\ref{negligible} as an application of
Theorem~\ref{main}.

\section{Projections}\label{results}

In this section we prove Theorem \ref{main} as a consequence of several lemmas.
In the case of orthogonal projections in
$\pr^n$, one can find a proof for the ``only if''-part of Theorem \ref{main} in
\cite[Chapter 18]{Ma}
or in \cite[Chapter 3.3]{F}. The main idea of our proof is same as that of
\cite{Ma}, but, due to our more
general setting, some modifications are naturally needed -- the major ones 
being in Lemma~\ref{E1-3}. 
For the convenience of the reader we give the main arguments.
In fact, our  
approach simplifies slightly the corresponding arguments in \cite{Ma}. 

In this section $\Lambda\subset\pr^l$ is open and $l,m$ and $n$ are integers 
with $m\le l$ and $m<n$. The closed ball with radius $r$ centred at $x$ is 
denoted by $B(x,r)$. As in \cite{Ma}, a non-negative, subadditive set function
vanishing for the empty set is called a measure. We start by defining cones 
around preimages of points with respect to Lipschitz continuous  mappings.

\begin{definition}\label{cones}
Let $\lambda\in\Lambda$ and let 
$P_\lambda:\mathbb R^n\to\mathbb R^m$ be Lipschitz
continuous. For all $a\in\mathbb R^n$, $0<s<1$ and $r>0$, we define
\begin{align*}
X(a,\lambda,s)& :=\{x\in\mathbb R^n\mid |P_\lambda(x)-P_\lambda(a)|
                  <s|x-a|\}{\textrm{ and }}\\
X(a,r,\lambda,s)& :=X(a,\lambda,s)\cap B(a,r).
\end{align*}
\end{definition}

The following lemma is an analogue of \cite[Corollary 15.15]{Ma}.

\begin{lemma}\label{E1}
Suppose that $E\subset\pr^n$ is purely $m$-unrectifiable. Let $\delta>0$ and
$\lambda\in\Lambda$. Defining 
\[
E_{1,\delta}(\lambda) :=\{a\in E\mid\limsup_{s\to 0}\sup_{0<r<\delta}(rs)^{-m}
  \mathcal H^m(E\cap X(a,r,\lambda,s))=0\},
\]
we have $\mathcal H^m(E_{1,\delta}(\lambda))=0$.
\end{lemma}

\begin{proof} Replacing $Q_V$ by $P_\lambda$ in 
\cite[Lemmas 15.13 and 15.14]{Ma} and observing that the Lipschitz constant
of $Q_V$ is one, the proof of \cite[Corollary 15.15]{Ma} 
works in our setting.
Here $Q_V$ is the projection onto the orthogonal complement $V^\perp$ of an 
$m$-plane going through the origin.
\end{proof}

Next we consider the analogue of \cite[Lemmas 18.3 and 18.4]{Ma} 
in our setting. The proof of \cite[Lemma 18.3]{Ma} relies on the fact that 
$Q_V(\{x\in B(a,r)\mid\vert Q_V(x-a)\vert<s\vert x-a\vert\})$
$=U(Q_V(a),rs)\cap V^\perp$ where $U(z,r)$ is the open ball with centre at $z$ 
and with radius $r$. 
Note that this does not hold when $Q_V$ is replaced by $P_\lambda$. 
However, the proof given in \cite[Lemma 3.3.9]{F} works in our setting. 

\begin{lemma}\label{E2}
Let $E\subset\pr^n$ with $\mathcal H^m(E)<\infty$, $\delta>0$ and
$\lambda\in\Lambda$. Defining
\[
E_{2,\delta}(\lambda) :=\{a\in E\mid\limsup_{s\to 0}\sup_{0<r<\delta}(rs)^{-m}
  \mathcal H^m(E\cap X(a,r,\lambda,s))=\infty\}
\]
and
\[
E_3(\lambda) :=\{a\in E\mid\#(E\cap P_\lambda^{-1}(P_\lambda(a)))=\infty\},
\]
we have 
$\mathcal H^m(P_\lambda(E_{2,\delta}(\lambda)))=0$ and
$\mathcal H^m(P_\lambda(E_3(\lambda)))=0$.
\end{lemma}

\begin{proof}
The first claim can be verified in the same way as \cite[Lemma 3.3.9]{F} 
and the latter one follows from \cite[Theorem 7.7]{Ma}. 
\end{proof}

Throughout the rest of this section we assume that the family 
$\{P_\lambda:\pr^n \to \pr^m\}_{\lambda\in\Lambda}$ is transversal. We use 
a slight variant of the $\beta=0$ case of the definition of 
$\beta$-transversality given in \cite[Definition 7.2]{PS}. 

\begin{definition}\label{transversality} Let $\Lambda\subset\pr^l$ be open. 
A family of maps $\{P_\lambda:\pr^n \to \pr^m\}_{\lambda \in \Lambda}$ 
is transversal if it satisfies the following 
conditions for each compact set $K\subset\pr^n$: 
\begin{itemize}
\item [(1)] The mapping $P:\Lambda\times K\to\pr^m$, 
  $(\lambda,x)\mapsto P_\lambda(x)$, is continuously differentiable and twice 
  differentiable with respect to $\lambda$.  
\item [(2)] For $j= 1,2$ there exist constants $C_j$ such that the derivatives 
  with respect to $\lambda$ satisfy 
  \[
   \Vert D_{\lambda}^j P(\lambda, x)\Vert \leq C_j \text{ for all }(\lambda,x) 
    \in\Lambda\times K. 
  \]
\item [(3)] For all $\lambda\in\Lambda$ and $x$, $y\in K$ with $x\ne y$, 
  define 
  \[
   T_{x, y}(\lambda) := \frac{P_\lambda(x) - P_\lambda(y)}{|x-y|}. 
  \]
  Then there exists a constant $C_T>0$ such that the property 
  \[
   |T_{x, y}(\lambda)|\leq C_T
  \] 
  implies that
  \[
   \det \left( D_\lambda T_{x, y}(\lambda) \left(D_\lambda T_{x, y}(\lambda)
    \right)^T \right) \geq C_T^2.
  \]
\item [(4)] There exists a constant $C_L$ such that
  \[
   \Vert D_{\lambda}^2 T_{x, y}(\lambda)\Vert\le C_L
  \]
  for all $\lambda\in\Lambda$ and $x,y\in K$ with $x \neq y$.
\end{itemize}
\end{definition}
 
Next we verify the analogue of \cite[Lemma 18.9]{Ma}.

\begin{lemma}\label{E1-3}
Let $E\subset\pr^n$ be $\mathcal H^m$-measurable with $\mathcal H^m(E)<\infty$
and  let $\delta>0$. For $\mathcal L^l$-almost all $\lambda\in\Lambda$ we have
for $\mathcal H^m$-almost all $a\in E$ either
\begin{align}
\limsup_{s\to 0}\sup_{0<r<\delta}(rs)^{-m}\mathcal H^m(E\cap X(a,r,\lambda,s))&=0
   \text{ or}\\
\limsup_{s\to 0}\sup_{0<r<\delta}(rs)^{-m}\mathcal H^m(E\cap  X(a,r,\lambda,s))&=
  \infty\text{ or}\\
(E\setminus\{a\})\cap P_\lambda^{-1}(P_\lambda(a))\cap B(a,\delta)\neq\emptyset.
\end{align}
\end{lemma}

\begin{proof}
Given $\delta>0$ and $a\in E$, we prove that for
$\mathcal L^l$-almost all $\lambda\in\Lambda$ either (2.1), (2.2) or (2.3) 
holds. Then the claim follows by Fubini's theorem. The measurability arguments
needed for applying Fubini's theorem are similar as those in 
\cite[Lemma 3.3.2]{F}. 
We may clearly
suppose that $E\subset K$ for some compact $K\subset\pr^n$, and furthermore, by
\cite[Theorem 1.10]{Ma} $E$ may be assumed to be $\sigma$-compact.

Fix $a\in E$, $\lambda_0\in\Lambda$ and $0<\delta<\delta_0$ such that
$B(\lambda_0,2\delta_0)\subset\Lambda$. Let $V\subset\pr^l$ be an
$m$-dimensional linear subspace and let $V_{\lambda_1}=V+\lambda_1$ for all
$\lambda_1\in\Lambda$. For all $\lambda_1\in B(\lambda_0,\delta_0)$,
define a measure $\Psi_{V_{\lambda_1}}$ on
$B(\lambda_0,2\delta_0)\cap V_{\lambda_1}$ by
\[
\Psi_{V_{\lambda_1}}(A) := \sup_{0<r<\delta}r^{-m}\mathcal H^m(E\cap B(a,r)\cap
  L_{V_{\lambda_1}}(A))
\]
for all $A\subset B(\lambda_0,2\delta_0)\cap V_{\lambda_1}$, where
\[
L_{V_{\lambda_1}}(A) :=\bigcup_{\lambda\in A}P_\lambda^{-1}(P_\lambda(a)).
\]

The set
\[
C_{V_{\lambda_1}} :=\{\lambda\in B(\lambda_0,2\delta_0)\cap V_{\lambda_1}\mid
   (E\setminus\{a\})\cap L_{V_{\lambda_1}}(\{\lambda\})\cap B(a,\delta)
   \neq\emptyset\}
\]
is $\mathcal H^m$-measurable. This follows from the
fact that it is $\sigma$-compact which can be seen as follows: defining
a continuous function
\[
\tilde P :\left(B(\lambda_0,2\delta_0)\cap V_{\lambda_1}\right)\times\pr^n\to
  \pr^m,\; \tilde P(\lambda,x) :=P_\lambda(x)-P_\lambda(a),
\]
and $\sigma$-compact sets
\[
S_1 :=\{(\lambda,x)\in\left(B(\lambda_0,2\delta_0)\cap V_{\lambda_1}\right)\x
  \pr^n\mid \tilde P(\lambda, x) = 0 \}
\]
and
\[
 S_2 :=S_1\cap\left(B(\lambda_0,2\delta_0)\x\left((E\setminus\{a\})\cap 
  B(a,\delta)\right)\right),
\]
we conclude that $C_{V_{\lambda_1}} = \Pi_\Lambda(S_2)$, where
$\Pi_\Lambda:\Lambda\x\pr^n\to\Lambda$ is the projection 
$\Pi_\Lambda(\lambda,x)=\lambda$. Thus $C_{V_{\lambda_1}}$ is 
$\sigma$-compact. 

Let 
$D_{V_{\lambda_1}} :=(B(\lambda_0,2\delta_0)\cap V_{\lambda_1})\setminus
   C_{V_{\lambda_1}}$.
From the definitions of $\Psi_{V_{\lambda_1}}$ and $C_{V_{\lambda_1}}$ we deduce 
that $\Psi_{V_{\lambda_1}}(D_{V_{\lambda_1}})=0$. Now
\cite[Theorem 18.5]{Ma} implies that for $\mathcal H^m$-almost all
$\lambda\in B(\lambda_0,\delta_0)\cap V_{\lambda_1}$ either
\begin{equation}\label{ad1}
 \limsup_{t\downarrow 0} t^{-m}\Psi_{V_{\lambda_1}}\left( B(\lambda_0,2\delta_0)
   \cap V_{\lambda_1}\cap B(\lambda, t) \right) = 0
\end{equation}
or
\begin{equation}\label{ad2}
 \limsup_{t\downarrow 0} t^{-m}\Psi_{V_{\lambda_1}}\left( B(\lambda_0,2\delta_0)
   \cap V_{\lambda_1}\cap B(\lambda,t) \right) = \infty
\end{equation}
or
\begin{equation}\label{ad3}
 \lambda \in C_{V_{\lambda_1}}.
\end{equation}

Applying Fubini's theorem we see that for $\mathcal L^l$-almost
all $\lambda\in B(\lambda_0,\delta_0)$ either \eqref{ad1}, \eqref{ad2} or
\eqref{ad3} holds with $V_{\lambda_1}$ replaced by $V_\lambda$.
(The measurability proofs needed here can be dealt with 
in a similar manner as those in \cite[Lemma 3.3.3]{F}.)
Note that here the exceptional set of $\mathcal L^l$-measure zero depends on
the $m$-plane $V$. Hence it is sufficient to find a finite collection 
of linear $m$-planes $V^1,\dots,V^k\subset\pr^l$ 
and $C>0$ such that for all $\lambda\in B(\lambda_0,\delta_0)$
\begin{align*}
\bigcup_{j=1}^k &B(a,r)\cap L_{V_\lambda^j}
(B(\lambda_0,2\delta_0)\cap V_\lambda^j\cap B(\lambda,C^{-1}s))\setminus\{a\}
  \subset X(a,r,\lambda,s)\\
&\subset\bigcup_{j=1}^kB(a,r)\cap L_{V_\lambda^j}
(B(\lambda_0,2\delta_0)\cap V_\lambda^j\cap B(\lambda,Cs))\setminus\{a\}
\end{align*}
for every small enough $s>0$. Indeed, by \cite[Lemma 3.3]{JJN} there are 
$C>0$ and $s_0>0$
such that for any $0<s<s_0$ and for any $x\in X(a,r,\lambda,s)$
there exists an $m$-dimensional coordinate plane $W$ such that
$x\in L_{W_\lambda}(B(\lambda_0,2\delta_0)\cap W_\lambda\cap B(\lambda,Cs))$, 
giving the latter inclusion for the 
collection of all $m$-dimensional coordinate planes in $\mathbb R^l$. Finally, 
the first
inclusion is true for any $m$-plane since, by transversality,
$\Vert D_\lambda T_{x,a}(\lambda)\Vert$ is bounded.
\end{proof}

For the ``if''-part of Theorem~\ref{main} we need the following lemma.

\begin{lemma}\label{rankm}
Assume that $\{P_\lambda:\pr^n\to\pr^m\}_{\lambda\in\Lambda}$ is
a transversal family of mappings. Then for every $a\in\mathbb R^n$, for
every $m$-dimensional $C^1$-submanifold $S\subset\mathbb R^n$ containing $a$ 
and for 
$\mathcal L^l$-almost all $\lambda\in\Lambda$ there exist $\gamma>0$ and $r>0$ 
such 
that $|P_\lambda(x)-P_\lambda(y)|\ge\gamma|x-y|$ for all $x,y\in B(a,r)\cap S$.
\end{lemma}

\begin{proof} 
We begin by showing that $P_\lambda$ is a submersion, that is,
$D_xP_\lambda(a)$ has rank $m$ at every point $a\in\mathbb R^n$. Here
$D_xP_\lambda$ is the derivative of $P_\lambda$
with respect to $x$.

Let $\lambda_0\in\Lambda$ and let $\ker D_xP_\lambda(a)\subset\mathbb R^n$ be 
the kernel of $D_xP_\lambda(a)$. 
By \cite[Lemma 3.2]{JJN}, Definition~\ref{transversality}  
implies that for any unit vector 
$e\in\ker D_xP_{\lambda_0}(a)$ one can find an $m$-dimensional plane 
$V^e\subset\mathbb R^l$ such that the mapping 
$g^e:V_{\lambda_0}^e\cap\Lambda\to\mathbb R^m$, defined as 
$g^e(\lambda) :=D_xP_\lambda(a)(e)$, is a diffeomorphism (onto its image) 
on a small neighbourhood of $\lambda_0$. Furthermore, the parallelepiped
$Dg^e(\lambda_0)([-1,1]^m)$ is uniformly thick -- by this we mean that 
the lengths of 
the edges and the angles between the 
edges are bounded from below by a constant which is independent of 
$\lambda_0\in\Lambda$, $e\in\ker D_xP_{\lambda_0}(a)$ and $a\in K$ for any fixed
compact $K\subset\mathbb R^n$. 

Since $D_xP_\lambda(a)$ is continuous in $\lambda$ and 
$\dim\ker D_xP_{\lambda}(a)\ge n-m$ for all $\lambda\in\Lambda$ there is
$e\in\ker D_xP_{\lambda_0}(a)$ such that 
$e=\lim_{\lambda\to\lambda_0}e_\lambda$,
where $e_\lambda\in\ker D_xP_{\lambda}(a)$. Define a function
$f^e:V_{\lambda_0}^e\cap\Lambda\to\mathbb R^n$ by 
\[
f^e(\lambda) := e-\proj_{\ker D_xP_\lambda(a)}(e), 
\]
where $\proj_V$ is the 
orthogonal projection onto $V\subset\mathbb R^n$. Observe that 
$g^e(\lambda)=D_xP_\lambda(a)(f^e(\lambda))$. The fact that 
$Dg^e(\lambda_0)([-1,1]^m)$
is uniformly thick implies that the same is true for 
$Df^e(\lambda_0)([-1,1]^m)$.

Assuming that $\dim\ker D_xP_{\lambda_0}(a)>n-m$ there are at most 
$m-1$ directions perpendicular to $\ker D_xP_{\lambda_0}(a)$.
Thus $Df^e(\lambda_0)([-1,1]^m)$ intersects $\ker D_xP_{\lambda_0}(a)$
in a set containing a line segment of positive length. In 
particular, there is a unit vector $v\in V^e$ satisfying 
$Df^e(\lambda_0)(v)\in\ker D_xP_{\lambda_0}(a)$ which, in turn, gives
the contradiction $Dg^e(\lambda_0)(v)=0$ and completes the proof that
$P_\lambda$ is a submersion. 
 
We proceed by verifying that for every 
$a\in\mathbb R^n$ and for every $m$-dimensional linear subspace
$W\subset\mathbb R^n$ we have $\ker D_xP_\lambda(a)\cap W=\{0\}$ for 
$\mathcal L^l$-almost all $\lambda\in\Lambda$. 

Fix $\lambda_0\in\Lambda$ such that $\ker D_xP_{\lambda_0}(a)\cap W=U$ with 
$\dim U=k$, where $1\le k\le m$. Clearly, it is sufficient to prove that
there is $\delta>0$ such that $\ker D_xP_\lambda(a)\cap W=\{0\}$ 
for $\mathcal L^l$-almost all $\lambda\in B(\lambda_0,\delta)$.
Let $e_1,\dots,e_k$ be an orthonormal basis for $U$ and let
$M :=\langle W\cup\ker D_xP_{\lambda_0}(a)\rangle$ be the subspace spanned by $W$
and $\ker D_xP_{\lambda_0}(a)$. Observe that $k=\dim M^\perp$.
For all $i=1,\dots,k$, consider the functions $f^{e_i}$ defined above. 
Since $P_\lambda$ is a submersion for all $\lambda$, we see that
$\ker D_xP_\lambda(a)$ tends to $\ker D_xP_{\lambda_0}(a)$ as 
$\lambda\to\lambda_0$. Thus $Df^{e_i}(\lambda_0)([-1,1]^n)$ is perpendicular to
$\ker D_xP_{\lambda_0}(a)$ for all $i=1,\dots,k$. 
In particular, for each $i$ there is a 
$k$-dimensional plane $W^{e_i}\subset V^{e_i}$ such that
$Df^{e_i}(\lambda_0)(W^{e_i})=M^\perp$. This implies the existence of
$v\in\mathbb R^l$ such
that $Df^{e_1}(\lambda_0)v,\dots,Df^{e_k}(\lambda_0)v$ are linearly independent.
Hence, for a sufficiently small $\varepsilon>0$ we have 
$\ker D_xP_\lambda(a)\cap W=\{0\}$ for $\mathcal L^l$-almost all  
$\lambda\in B(\lambda_0,\varepsilon)\cap \langle v\rangle_{\lambda_0}$. 
By continuity, there
exists $\delta>0$ such that this is valid if we replace $\lambda_0$ by 
any $\lambda_1\in B(\lambda_0,\delta)$. Finally, Fubini's theorem implies that
$\ker D_xP_\lambda(a)\cap W=\{0\}$ for $\mathcal L^l$-almost all  
$\lambda\in B(\lambda_0,\delta)$.

The claim follows by choosing $W=T_aS$ and using the fact that since 
$P_\lambda$ is a smooth submersion it is locally a fibration (see 
\cite[Remark 1.92]{GHL}).
\end{proof}

Now we are ready to prove the generalization of the 
Besicovitch-Federer projection theorem.

\begin{proof}[Proof of Theorem \ref{main}]
The proof of the ``only if''-part of Theorem \ref{main} is similar to the
one given in \cite[p. 257-258]{Ma}. Indeed, defining $E_{1,\delta}(\lambda)$
and $E_{2,\delta}(\lambda)$ as in Lemmas~\ref{E1} and \ref{E2}, setting
\[
E_{3,\delta}(\lambda) := \{ a \in E \mid (E \setminus \{a\}) \cap P_\lambda^{-1}
  (P_\lambda(a)) \cap B(a, \delta) \neq \emptyset \},
\]
and applying Lemmas~\ref{E1}, \ref{E2} and \ref{E1-3}, we conclude,
as in \cite[p. 257-258]{Ma}, that the claim holds. 
 
To prove the ``if''-part of the theorem, assume to the contrary that 
there is an $m$-rectifiable
$F\subset E$ with $\mathcal H^m(F)>0$. According to 
\cite[Theorem 3.2.29]{F}, there exist $m$-dimensional $C^1$-submanifolds 
$S_1,S_2,\ldots\subset\mathbb R^n$ such that 
$\mathcal H^m(F\setminus\cup_{i=1}^\infty S_i)=0$. 

Fixing $i$ and letting $a$ be 
a density point of $F\cap S_i$, Lemma~\ref{rankm} implies the existence of
$\gamma>0$ and $r>0$ such that for $\mathcal L^l$-almost all 
$\lambda\in\Lambda$ we have
$\vert P_\lambda(x)-P_\lambda(y)\vert\ge\gamma\vert x-y\vert$ for all
$x,y\in B(a,r)\cap S_i$.
This in turn gives that 
$\mathcal H^m(P_\lambda(E))\ge\gamma^m\mathcal H^m(F\cap B(a,r)\cap S_i)>0$
for $\mathcal L^l$-almost all $\lambda\in\Lambda$ which is a 
contradiction.
\end{proof}

\begin{remark}\label{lipschitz}
In the ``only if''-part of the previous proof we did not use the assumption 
that the mapping $(\lambda,x)\mapsto P_\lambda(x)$ is continuously 
differentiable in $x$ 
(see Definition~\ref{transversality}). It is sufficient to suppose that 
it is Lipschitz 
continuous. The differentiability in the second coordinate is needed only for 
the ``if''-part of Theorem~\ref{main}.
\end{remark}

\section{Dynamics of the geodesic flow}\label{application}

\subsection{Pairs of pants and right angle octagons}\label{octagons}

The contents of this subsection and the following one are standard, see e.g.
\cite{Se}. Suppose $S$ is a pair of pants with cuff lengths $a, b$ and $c$ (See Figure
\ref{labelling}).
The {\it {seams}} of $S$ are the shortest geodesic segments connecting the
cuffs. Consider the seam connecting the cuffs $a$ and $c$ and code by $\be$ and
$\ob$ the two sides of this seam. Analogously, consider the seam connecting the
cuffs $b$ and $c$ and code by $\al$ and $\oa$ its two sides. If we cut $S$
along these two seams, we obtain a hyperbolic octagon with right angles. We
label the four sides of this octagon corresponding to the cut seams by the code
of the part of $S$ inside the octagon. The $c$ cuff is cut into two geodesics
of length $c/2$, which we label as $c_1$ and $c_2$. We see consecutively the
labels $\al$, $b$, $\oa$, $c_1$, $\ob$, $a$, $\be$ and $c_2$ on the sides of 
the octagon (up to possibly exchanging the role of $\al$ and $\oa$, $\be$ and 
$\ob$, or $c_1$ and $c_2$).
Let $R$ be a copy of the octagon inside the hyperbolic space $\mathbb H^2$. For
$\tau=\al$, $\oa$, $\be$ or $\ob$, let $\varphi_\tau$ be the M\"obius
transformation sending the geodesic $\tau$ on the geodesic $\overline\tau$
(with the convention that $\overline{\overline\tau} = \tau$) and the
half-plane separated by the complete extension of $\tau$ containing $R$ onto
the half-plane separated by the complete extension of $\overline\tau$ not
containing $R$. We have $\varphi_{\overline\tau}=\varphi_\tau^{-1}$ for all
$\tau$. The union of $S$ and its boundary $\partial S$ is obtained from the 
closure of $R$
by identifying the sides $\al$ and $\oa$ using $\varphi_\al$ and by identifying
$\be$ and $\ob$ using $\varphi_\be$. Moreover, the geodesics extending the
$\tau$ sides do not intersect one another, and therefore, 
by the classical ping-pong
argument, $\varphi_\al$ and $\varphi_\be$ generate a free group $G$. The
images of the interior of $R$ by $G$ are disjoint and the region containing $R$ and
delimited by the four extensions of the $\tau$ geodesics is a fundamental domain for
$G$. For all $g\in G$, label the geodesic sides of $gR$ by the image of the
labelling of the geodesic sides of $R$. In a consistent way, each geodesic 
segment of the form $g\tau$ has two opposite labels corresponding to the two 
images of $R$ that it separates.

\begin{figure}
 \begin{center}
 \input{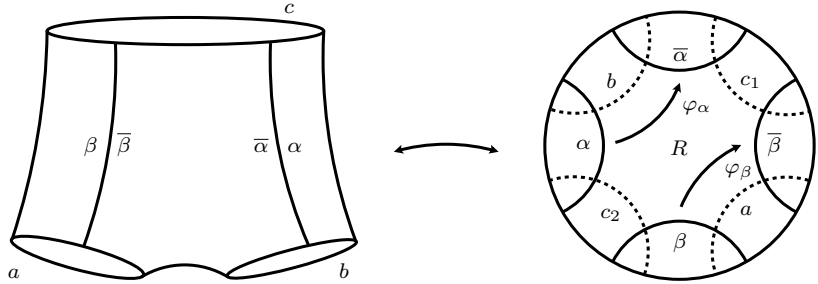}\\
 \centering\caption{Pair of pants and the labelling of the sides of the octagon $R$.}
   \label{labelling}
 \end{center}
\end{figure}

We say that a geodesic $\gamma$ in $T^1\mathbb H^2$ starts from $R$ if 
$\gamma(0)\in\partial R$ and there is some $t>0$ with $\gamma(t)\in R$. 
Let $\ga$ be a geodesic starting from $R$. It corresponds to
a geodesic in $C(S)$ (recall the definition \eqref{PiNW}), if and only if it
never cuts the sides of $G(R)$ labelled
as $a$, $b$, $c_1$ or $c_2$. In other words, $\ga$ intersects only $\tau$
geodesics. Record the interior label of these geodesics successively as
$\om_n$, $n\in\mathbb Z$, $\om_0$ being the label of the side by which the
geodesic $\ga$ enters $R$. This sequence is called the {\it{cutting sequence}}
of $\ga$. The cutting sequence of any geodesic in $C(S)$ is a reduced infinite
word in $\al$, $\oa$, $\be$ and $\ob$, where reduced means that the succession
$\tau\overline\tau$ is not permitted. Since two infinite geodesics in
$\mathbb H^2$ with distinct supports are not at a bounded distance from each
other, any cutting sequence is the cutting sequence of a unique geodesic. The
boundary geodesics correspond to the reduced words $(\al)^\infty$,
$(\oa)^\infty$, $(\be )^\infty$, $(\ob )^\infty$, $(\al\ob)^\infty$ and
$(\oa\be)^\infty$.
%In Subsection \ref{Markov} we observe that every other
%infinite reduced word is associated to a unique geodesic in $C(S)$.

Consider the four disjoint complete geodesics in $\mathbb H^2$ extending the
segments $\al, \oa, \be$ and  $\ob$ of the previous subsection. Each of them
cut $S^1$, the circle at infinity, into two intervals. Write $A$,
$\overline A$, $B$ and $\overline B$ for the interval separated from $R$ by the
geodesic $\al$, $\oa$, $\be$ and $\ob$, respectively. Let $\varphi$ be defined
on each $T\;(T= A,\overline A,B,\overline B)$ by the corresponding M\"obius
transformation $\varphi_\tau$. The mapping $\varphi$ is expanding (see
\cite{Se}) and $\varphi(T) = S^1\setminus\Int\overline T$, where the interior of a set
in $S^1$ is denoted by $\Int$. In particular,
$\varphi(T)$ contains the three intervals different from $\overline T$.

We define the boundary expansion of a point $\xi\in S^1$. If $\xi$ does not
belong to $\Int(A\cup\overline A\cup B\cup\overline B)$, stop here. Otherwise,
let $\xi_0=\al$, $\oa$, $\be$ or $\ob$ accordingly. Apply then the procedure to
$\varphi(\xi)$ and iterate. Every point has an empty, finite or infinite
sequence of symbols attached, which is called its {\it{boundary expansion.}}
Boundary expansions are reduced words in $\al$, $\oa$, $\be$ and $\ob$. The set
of points with an infinite boundary expansion is a Cantor subset
$\Om\subset S^1$.
For a geodesic $\ga$ starting in $R$, the positive part of the coding sequence
is the boundary expansion  of the limit point $\ga (+\infty)$. Similarly, the
sequence $\overline{\om_0},\overline{\om_{-1}},\overline{\om_{-2}},\dots$ is
the boundary expansion of $\ga(-\infty )$. This defines a one-to-one
correspondence $\Psi$ between cutting sequences of geodesics starting from $R$
and the set
\[
(\Om\x\Om)^\ast=\{(\xi,\eta)\in\Om\x\Om\mid\xi_0\ne\eta_0\},
\]
namely, $\Psi(\om)=(\xi,\eta)$ where $\xi_i=\om_{i+1}$ and
$\eta_j=\overline{\om_{-j}}$ for $i,j=0,1,\dots$

Clearly, if $\{\om_n\}_{n\in\mathbb Z}$ is the cutting sequence of the geodesic
$\ga$, the shifted sequence $\{\om_n'\}_{n\in\mathbb Z},\om_n' = \om_{n+1}$ is
associated to the geodesic $\ga(\cdot + \ell)$, where  $\ell$ is the first
positive time $t$ when $\ga(t)$ is not in $R$. Consider the mapping
\[
\Phi : \{(\xi,\eta,s)\mid (\xi,\eta)\in(\Om\times\Om)^\ast,
  0\le s<\ell(\Psi^{-1}(\xi,\eta))\}\longrightarrow T^1\mathbb H^2
\]
which associates to $(\xi,\eta,s)$ the point $(x,v)\in T^1\mathbb H^2$ such
that the geodesic $\ga$ with initial condition $(x,v)$ satisfies
$\ga(+\infty)=\xi$, $\ga(-\infty)=\eta$ and
$\ga(-s)$ is entering into $R$. The mapping $\Phi$ is a restriction of the
usual chart of $T^1\mathbb H^2$ given by $(S^1\times S^1)^\ast\times\mathbb R$.
Its image is a subset of $T^1R$ which is identified with
$NW=D(S)\cup T^1(\partial S)$ (recall \eqref{NW}). Metric properties of $NW$,
and consequently those of
$C(S)$, will be read from metric properties of $\Omega$ through this
Lipschitz mapping $\Phi$. Moreover, from the above symbolic representation, we
see that $NW$ is the nonwandering set of the geodesic flow on
$T^1S\cup T^1(\partial S)$. The geodesic flow,
restricted to $D(S)\cup T^1(\partial S)$, is therefore
represented by a suspension over the set of reduced words with suspension
function $\ell(\om)$, where $\ell(\om)$ is the time spent in $R$ by the
geodesic with cutting sequence $\om$.

\subsection{Markov repellers}\label{Markov}

We use properties of Markov repellers as established by Bowen and Ruelle
\cite{R1}. A Markov repeller is an expanding piecewise $C^{1+\alpha}$ map of the
real line into itself with a finite family of disjoint intervals $A_i, i\in J$,
such that if $f(A_i)$ intersects the interior of some $A_j$, then $f(A_i)$
contains $A_j$. The set of points which remain in $\cup_{j\in J}A_j$ under
applications of all the iterates $f^n$, $n\in\mathbb N$, is a Cantor set $X$.
The set $X$ is invariant under $f$. For any $f$-invariant probability measure
$\mu$ on $X$ consider the metric entropy $h_\mu(f)$. For any continuous
function $g$ on $X$, define the {\it {pressure}} $P(g)$ by
\[
P(g):=\sup_\mu\left\{h_\mu(f) + \int_X g\,d\mu\right\},
\]
where $\mu$ varies over all $f$-invariant probability measures on $X$. Assume
that $f$ is topologically transitive. Then there exists a unique $s$ with
$0<s<1$ such
that $P(-s\ln|f'|)=0$. The number $s$ is both the Hausdorff dimension and the
packing dimension of $X$. More precisely, there exists a unique $f$-invariant
probability measure $\mu_0$ on $X$ such that
\[
h_{\mu_0}-s\int_X\ln|f'|\,d\mu_0=0.
\]
The measure $\mu_0 $ is Ahlfors $s$-regular on $X$: for all $\varepsilon $
small enough and for all $x\in X$ the ratio
$\mu(B(x,\varepsilon))\varepsilon^{-s}$ is bounded away from 0 and infinity. In
particular, $0<\mathcal H^s(X)<\infty$.

Since the Patterson measure $\nu_0 $ is also Ahlfors regular \cite[Section 3]{S1}, the
measures $\nu _0 $ and $\mu _0 $ are mutually absolutely continuous with bounded densities.
The geodesic flow invariant measure $m$ constructed in \cite[Section 4]{S1} (called the
\linebreak
Bowen-Margulis-Patterson-Sullivan measure) has support $D(S)$, is the measure of maximal
entropy $s$ for the geodesic flow on $T^1S$ and has dimension $1 + 2s$.

Finally, if $(a,b,c)\mapsto f_{a,b,c}$ is a real analytic family of piecewise
$C^{1+\alpha}$ expanding mappings, then the function $(a,b,c)\mapsto\dimH(X)$ is
real analytic as well (see for example
\cite[Corollary 7.10 and Section 7.28]{R2}).

\subsection{Proof of Theorem \ref{negligible}}

For fixed $a,b$ and $c$, consider the set $\Om_{a,b,c}\subset S^1$ of the
previous subsection. It is a transitive Markov repeller for the mapping
$\varphi_{a,b,c}$. The mapping $\varphi_{a,b,c}$ is given by a piecewise
M\"obius transformation, and therefore, it belongs to 
a semi-algebraic variety of
piecewise analytic mappings. Moreover, $(a,b,c)\mapsto\varphi_{a,b,c}$ is real
analytic, and thus the function $(a,b,c)\mapsto\dd(a,b,c)=\dimH(\Om_{a,b,c})$
is real analytic. In particular, there is a two-dimensional submanifold of
values $a,b$ and $c$ such that $\dd(a,b,c)=1/2$.

\begin{proposition}\label{2dim}
Assume that $\dd(a,b,c)=1/2$. Then the nonwandering set $NW$ is purely
2-unrectifiable and has positive and finite 2-dimen\-sional Hausdorff measure.
\end{proposition}

\begin{proof}
It is enough to consider $D(S)$ since $T^1(\partial S)$ is 1-dimensional.
Recalling that $\tau\overline\tau$ is a forbidden word for $\xi\in\Omega$, the
above discussion implies that $D(S)=\cup_{i=1}^nU_i$ and each $U_i$ is Lipschitz
equivalent to an open subset of $\Om\x\Om\x I$, where $I$ is a real interval.
Since the measure $\mu_0$ is Ahlfors $1/2$-regular on $\Om$, the measure
$\mu_0\x\mu_0\x\mathcal L^1$ is Ahlfors $2$-regular on $\Om\x\Om\x I$.
Therefore $\dimH(\Om\x\Om\x I)=2$ and $0<\mathcal H^2(\Om\x\Om\x I)<\infty$.
Thus $\dimH(D(S))=2$ and $0<\mathcal H^2(D(S))<\infty$. For the first claim it 
is enough to notice that $\Om\times\Om$ is purely $1$-unrectifiable, since it 
is a product of two Cantor sets of dimension $1/2$ \cite[Example 15.2]{Ma}. Thus the product
$\Om\times\Om\times I$ is purely $2$-unrectifiable, and so is $D(S)$.
\end{proof}

Now we are ready to prove Theorem \ref{negligible}.

\begin{proof}[Proof of Theorem \ref{negligible}]
In \cite[Section 3]{JJL} it is shown that locally there exist an open set
$U\subset T^1S$, bi-Lipschitz mappings $\psi_1:U\to I^3$ and
$\psi_2:I^2\to\Pi(U)$ and a smooth mapping $P:I^3\to I^2$ such that
\[
 \Pi|_U=\psi_2\circ P\circ\psi_1,
\]
where $I\subset\pr$ is the open unit interval. The mapping $P$ is defined by
$P(y_1,y_2,t)=(P_t(y_1,y_2),t)$, where $\{P_t:I^2\to I\}_{t\in I}$ is a
transversal family of smooth mappings.

By Proposition~\ref{2dim}, the set $\psi_1(D(S)\cap U)=E\times I$ is purely
$2$-unrectifiable. Thus $E\subset I^2$ is purely $1$-unrectifiable. 
Furthermore,
$P(E\times I) = \bigcup_{t\in I}P_t(E)\times\{t\}$. By Theorem \ref{main},
\[
\hm^1(P_t(E))=0\text{ for }\mathcal L^1\text{-almost all }t\in I,
\]
giving $\hm^2(P(E \times I)) = 0$ by Fubini's theorem. This implies that
\[
\begin{aligned}
 \hm^2(\Pi(D(S)\cap U))&=\hm^2((\psi_2\circ P\circ\psi_1)(D(S)\cap U))\\
 &=\hm^2(\psi_2(P(E\times I)))=0,
\end{aligned}
\]
since $\psi_2$ is a bi-Lipschitz mapping. The claim follows from the fact that 
$T^1S$ can be covered by countably many open sets $U$.
\end{proof}

\begin{corollary}
 The Bowen-Margulis-Patterson-Sullivan measure $m$
is 2-dimensional, its support $\spt m = D(S)$ is 2-dimensional and
$\mathcal L^2(\spt(\Pi_*m))=0$.
\end{corollary}

\end{document}